\documentclass[11pt]{article}
\usepackage{amsmath}
\usepackage{bbm}
\usepackage{amsmath,amsthm,amssymb,amscd}
\usepackage{latexsym}
\usepackage{hyperref}
\usepackage{indentfirst}
\usepackage[numbers,sort&compress]{natbib}

\newtheorem{theorem}{Theorem}[section]
\newtheorem{corollary}{Corollary}[section]
\newtheorem{lemma}{Lemma}[section]

\newtheorem{remark}{Remark}[section]

\numberwithin{equation}{section} 

\title{\Large \bf Liouville-type theorems for the fourth order nonlinear elliptic equation\thanks{The work was partially supported
by NSFC of China (No. 11201248), K.C. Wong Fund of Ningbo University and Ningbo Natural Science Foundation (No. 2012A610031).}}
\author{{\small Liang-Gen Hu\footnote{email address: hulianggen@tom.com}}\\[0.05cm] {\small Department of Mathematics,
Ningbo University, 315211, P.R. China}}
\date{}
\begin{document}
\maketitle
\def\abstractname{}
\vspace*{-35pt}

\begin{abstract}
\noindent {\bf Abstract:} In this paper, we are concerned with Liouville-type theorems for the nonlinear
elliptic equation
\begin{equation*}
\Delta^2 u=|x|^a |u|^{p-1}u\;\ \mbox{in}\;\ \Omega,
\end{equation*}where $a \ge 0$, $p>1$ and $\Omega \subset \mathbb{R}^n$ is an unbounded domain of $\mathbb{R}^n$, $n \ge 5$.
We prove Liouville-type theorems for solutions belonging to one of the following classes: stable solutions and
finite Morse index solutions (whether positive or sign-changing). Our proof is based on a combination of the
{\it Pohozaev-type identity}, {\it monotonicity formula} of solutions and a {\it blowing down} sequence,
which is used to obtain sharper results.
\\ [0.2cm]
\end{abstract}\vskip .2in

\section{Introduction}

The paper is devoted to the study of the following nonlinear fourth order elliptic equation
\begin{equation}\label{eq:1.1}
\Delta^2 u=|x|^a |u|^{p-1}u\;\ \mbox{in}\;\ \Omega
\end{equation}where $a \ge 0$, $p>1$ and $\Omega \subset \mathbb{R}^n$ is an unbounded domain of $\mathbb{R}^n$, $n \ge 5$.
We are interested in the Liouville-type theorems---i.e., the nonexistence of the solution $u$ which is stable or
finite Morse index, and the underlying domain $\Omega$
is an arbitrarily unbounded domain of $\mathbb{R}^n$.\vskip .1in

The idea of using the Morse index of a solution of a semilinear elliptic equation
was first explored by Bahri and Lions \cite{Bahri} to obtain further
qualitative properties of the solution. In 2007, Farina \cite{Farina} made significant progress,
and considered the Lane-Emden equation
\begin{equation}\label{eq:1.2}
- \Delta u =|u|^{p-1}u\;\ \mbox{in}\;\ \Omega,
\end{equation}on bounded and unbounded domains of $\Omega \subset \mathbb{R}^n$, with $n \ge 2$ and $p >1$. Farina completely
classified finite Morse index solutions (positive or sign-changing) in his seminal paper \cite{Farina}.
His proof makes a delicate application of the classical Moser iteration method. There exist many
excellent papers to use the generalization of Moser's iteration technique to discuss the harmonic and fourth-order
elliptic equation. We refer to \cite{Dancer-1,Hajlaoui,Wang,Wei,Wei-Ye}
and the reference therein.

However, the classical Moser's iterative technique may fail to obtain the similarly
complete classification for the biharmonic equation
\begin{equation}\label{eq:1.3}
\Delta ^2 u=|u|^{p-1}u\;\ \mbox{in} \;\ \Omega \subset \mathbb{R}^n.
\end{equation}
Recently, D\'{a}vila, Dupaigne, Wang and Wei \cite{Davila} have
derived a monotonicity formula for solutions of (\ref{eq:1.3}) to reduce the nonexistence of nontrivial entire solutions
for the problem (\ref{eq:1.3}), to that of nontrivial homogeneous solutions, and gave a complete classification of
stable solutions and those of finite Morse index solutions. We note that Pacard \cite{Pacard-1,Pacard-2}
studied the partial regularity results for stationary weak solution of $- \Delta u=u^p$ by the
use of monotonicity formula.\vskip .1in

Let us recall that for the Liouville-type theorems and properties of the {\bf subcritical case} has been extensively studied by many authors. Gidas and Spruck have been
investigated the optimal Liouville-type theorems in the celebrated paper \cite{Gidas}. Thus, the equation
(\ref{eq:1.2}) has no positive solution if and only if
\begin{equation*}
p<\dfrac{n+2}{n-2}\ (=+\infty,\ \ \mbox{if}\;\ n \le 2).
\end{equation*}The {\bf supercritical case} $p>\dfrac{n+2}{n-2}$  is much less
complete understood. Bidaut-V\'{e}ron and V\'{e}ron \cite{Bidaut} proved the asymptotic behavior of positive solution of (\ref{eq:1.2})
by the use of the Bochner-Lichnerowicz-Weitzenb\"{o}ck formula in $\mathbb{R}^n$.\vskip .1in

On the other hand, that the understanding of the {\bf case $a \neq 0$} is less complete and is more delicate
to handle than the case $a=0$. In \cite{Gidas},
Gidas and Spruck concluded that for $a \le -2$, the equation
\begin{equation}\label{eq:1.4}
- \Delta u=|x|^a u^p, \; \mbox{in}\;\ \Omega
\end{equation}has no positive solution in any domain $\Omega$ containing the origin. Recently,
Dancer, Du and Guo \cite{Dancer-1} have researched the asymptotical behavior of stable and finite Morse
index solutions of (\ref{eq:1.4}), where $a>-2$ and $p<\overline{p}(a^-)$ ($\overline{p}(0)$), $a^-=\min\{0,a\}$.

The {\bf case} $a>0$ seems some difficult. Since the classical techniques and many properties may fail to deal with the corresponding equations.
In 2012,  Phan and Souplet \cite{Phan} used the delicate method in \cite{Souplet}
to prove that if $n \ge 2$, $a>0$, $1<p<\dfrac{n+2+2a}{n-2}$ and $n=3$, then the equation (\ref{eq:1.4})
has no positive bounded solution in $\Omega=\mathbb{R}^n$. Meanwhile, adopting the similar method, Fazly and Ghoussoub proved the following
result:\vskip .1in

\noindent {\bf Theorem A.} (\cite[Theorem 3]{Falzy}) {\it Let $n \ge 5$, $a \ge 0$ and $p>1$. Then for any Sobolev subcritical exponent, i.e.,
\begin{equation*}
1 < p < \dfrac{n+4+2a}{n-4},
\end{equation*}the equation (\ref{eq:1.1}) has no positive solution with finite Morse index.
}\vskip 0.08in

Inspired by the ideas in \cite{Davila,Pacard-1}, our purpose in this paper is to prove the Liouville-type theorems
in the class of stable solution and finite Morse index solution. Thus for any fixed $a \ge 0$ and $n\ge 3$, we get
\vskip .1in

\begin{theorem}\label{eq:t1.1}
If $u$ is a smooth stable solution of (\ref{eq:1.1}) in $\mathbb{R}^n$ and $1<p<p_a(n)$,
then $u \equiv 0$.
\end{theorem}

\begin{theorem}\label{eq:t1.2}
Let $u$ be a smooth solution to (\ref{eq:1.1}) with finite Morse index.
\begin{itemize}
\item If $p \in (1, p_a(n))$, $p \neq \dfrac{n+4+2a}{n-4}$, then $u \equiv 0$.
\item If $p=\dfrac{n+4+2a}{n-4}$, then $u$ has finite energy, i.e.,
\begin{equation*}
\int_{\mathbb{R}^n} \left (\Delta u \right )^2 =\int_{\mathbb{R}^n} |x|^{\alpha} |u|^{p+1} <+\infty.
\end{equation*}
\end{itemize}
\end{theorem}\vskip .1in

\noindent Here the representation of $p_a(n)$ in Theorem \ref{eq:t1.1} and \ref{eq:t1.2} is given by (\ref{eq:2.01}) below.\vskip .1in

\begin{remark}
\begin{itemize}
\item [\rm (1)]  Let us note that for any fixed $a \ge 0$ and
$\dfrac{n+4+4a}{n-4}<p<p_a(n)$, we adopt a new method---a combination of monotonicity formula and blowing down sequence---
to deal with the case and get Liouville-type theorem.
\item [\rm (2)]For the subcritical and critical cases, the proof is based on the combination of the {\it Pohozaev
identity} with some integral and pointwise estimates obtained by the doubling lemma in \cite[Lemma 5.1]{Polacik}.
\item [\rm (3)] In contrast with the results of \cite{Falzy}, our result is extended to the larger interval
$\left (1,p_a(n) \right )$ and the proof of method is different and independent interesting. For equation (\ref{eq:1.1}),
we do not impose any sign condition for $u$ and extra restrictions on $n$, $a$ and $p$.
\end{itemize}
\end{remark}

To describe our results more accurately, we need to make precise several terminologies.

\begin{itemize}
\item {\bf Definition.}\hspace*{5pt} We recall that a critical point $u \in C^2 (\Omega)$ of the energy functions
\begin{equation*}
\mathcal{L}(u)=\int_{\Omega} \dfrac{1}{2} |\Delta u|^2dx -\dfrac{1}{p+1} \int_{\Omega} |x|^a |u|^{p+1}dx
\end{equation*}is said to be
\begin{itemize}
\item [\rm (i)] a stable solution of (\ref{eq:1.1}), if for any $\psi \in C_0^4( \Omega )$, we have
\begin{equation*}
\mathcal{L}_{uu}(\psi):=\int_{\Omega} |\Delta \psi|^2dx - p\int_{\Omega} |x|^a |u|^{p-1} \psi^2 dx \ge 0.
\end{equation*}
\item [\rm (ii)] a solution $u$ of (\ref{eq:1.1}) with a Morse index equal to $l \ge 0$, if $l$ is the maximal dimension of
a subspace $X_l$ of $C_0^1(\Omega)$ such that $\mathcal{L}_{uu}(\psi) <0$ for all $\psi \in X_l \backslash \{0\}$. Therefore, $u$
is stable if and only if its Morse index is equal to zero.
\item [\rm (iii)] a stable solution $u$ of (\ref{eq:1.1}) outside a compact set $\Gamma \subset \Omega$, if $\mathcal{L}_{uu}
(\psi) \ge 0$ for any $\psi \in C_0^1(\Omega \backslash \Gamma )$. It follows that any finite Morse index
solution $u$ is stable outside some compact set $\Gamma \subset \Omega$.
\end{itemize}
\item {\bf Notation.}\hspace*{5pt} Here and in the following, we use $B_r(x)$ to denote the open ball on $\mathbb{R}^n$ central at $x$
with radius $r$. we also write $B_r=B_r(0)$. $C$ denotes various irrelevant positive constants.
\end{itemize}

The organization of rest of the paper is as follows. In section 2, we construct a monotonicity formula which is a crucial tool
to handle the supercritical case,
and derive various integral estimates. Then we prove Liouville-type theorem for stable solutions of (\ref{eq:1.1}), this is
Theorem \ref{eq:t1.1} in Section 3. To prove the result, we first obtain the nonexistence of homogeneous, stable solution of (\ref{eq:1.1})
in $\mathbb{R}^n \backslash \{0 \}$, where $p$ belongs to $\left ( \dfrac{n+4+2a}{n-4}, p_a(n) \right )$ (the representation of $p_a(n)$
in the below (\ref{eq:2.01})). Secondly, we obtain some estimates of solutions, and show that
the limit of blowing down sequence $u^{\infty}(x)=\lim\limits_{\tau \to \infty} \tau^{\frac{4+a}{p-1}} u(\tau x)$
satisfies $E(r;0,u) \equiv  \mbox{const}$. Here, we use the monotonicity formula of Theorem \ref{eq:t2.1}.
In Section 4, we study Liouville-type theorem of finite Morse index solutions by the use of the {\it Pohozaev-type
identity}, {\it monotonicity formula} and {\it blowing down} sequence.

\section{A Monotonicity formula and some estimates}

In this section, we construct a monotonicity formula which play an important role in dealing with the supercritical case,
and obtain various integral estimates of stable solutions.

To explore the main results in this paper, we need to define a {\it critical power} of (\ref{eq:1.1}).
For any fixed $a \ge 0$ and $n \ge 5$, we define the functions by
\begin{align*}
g(p) : & = p \left (\dfrac{4+a}{p-1}+2 \right )\left (n-4-\dfrac{4+a}{p-1} \right )+p\dfrac{4+a}{p-1} \left (n-2-\dfrac{4+a}{p-1} \right ),\\[0.09cm]
f(p) : & = p \dfrac{4+a}{p-1}\left (\dfrac{4+a}{p-1}+2 \right ) \left (n-4-\dfrac{4+a}{p-1}
\right ) \left (n-2-\dfrac{4+a}{p-1} \right ).
\end{align*}A direct computation finds
\begin{align*}
g \left (\dfrac{n+4+2a}{n-4} \right ) & =\dfrac{n+4+2a}{n-4}\times \dfrac{n(n-4)}{2} >\dfrac{n(n-4)}{2},\\[0.09cm]
f \left (\dfrac{n+4+2a}{n-4} \right ) & =\dfrac{n+4+2a}{n-4}\times \dfrac{n^2(n-4)^2}{16}>\dfrac{n^2(n-4)^2}{16}.
\end{align*}and differentiating the function $f(p)$ in $p$, we get
\begin{align*}
f'(p) = & 2p\dfrac{(4+a)^2}{(p-1)^3} \left (\dfrac{4+a}{p-1}+2 \right ) \left [ n-3-\dfrac{4+a}{p-1}\right ] \\[0.06cm]
& -\dfrac{4+a}{(p-1)^2}\left (6+a+\dfrac{8+2a}{p-1} \right ) \left (n-4-\dfrac{4+a}{p-1} \right )\left (n-2-\dfrac{4+a}{p-1} \right ).
\end{align*}It is easy to check that
\begin{equation*}
f'\left (\dfrac{n+4+2a}{n-4} \right )=\dfrac{n^2(n-4)^2}{16}>0.
\end{equation*}

Let $n(a)$ be the integer part of the largest real root of the algebra equation
\begin{equation*}
x^3-4x^2-32(a+4)x+64a+256=0,
\end{equation*}and $p(n,a)$ be the largest real root of the algebra equation
\begin{align*}
&\Big [ n^4-8n^3-16(2a+7)n^2+192(a+4)n-256(a+4)\Big ] x^4\\[0.06cm]
& -4 \Big [ n^4 + 8 n^3+ 4\big(a^2+ 2 a - 4\big) n^2 - 8\big( 5a^2 + 22 a + 8\big)n+16\big(5 a^2 + 28 a + 32 \big) \Big ] x^3 \\[0.06cm]
& +2 \Big [ 3 n^4- 24 n^3+ 16\big(a^2 + 5 a +7\big) n^2+ 16\big(a^3 + 2 a^2 - 14 a -24\big) n -64 \big(a^3 + 7 a^2+14 a +8\big) \big ] x^2 \\[0.06cm]
& -4 \Big [ n^4-8 n^3 + 4\big(a^2 + 6 a +12\big) n^2 + 8\big(a^3 + 7 a^2+ 14 a + 8\big) n +4a\big(a^3 + 8 a^2 + 20 a +16\big) \Big ] x\\[0.06cm]
& +n^4-8 n^3+16 n^2=0.
\end{align*}For any fixed $a \ge 0$ and $n \ge 5$, we define
\begin{equation}\label{eq:2.01}p_a(n)=
\begin{cases}
+\infty, & \mbox{if}\;\ n \le n(a),\\
p(n,a), & \mbox{if}\;\ n \ge n(a)+1.
\end{cases}
\end{equation}Therefore, we find
\begin{equation*}
f(p) >\dfrac{n^2(n-4)^2}{16},
\end{equation*}
for any $\dfrac{n+4+2a}{n-4} <p <p_a(n)$.

In particular, if $a=0$, then $p_0(n)$ in (\ref{eq:2.01}) is the fourth order Joseph-Lundgren exponent which
is computed by Gazzola and Grunau \cite{Gazzola}.

Furthermore, using the inequality $x+y\ge 2\sqrt{xy}$, for all $x,y\ge 0$,
and combining with the definition of the functions $g(p)$ and $f(p)$, we obtain
\begin{equation*}
g(p)>\dfrac{n(n-4)}{2},
\end{equation*}for any $\dfrac{n+4+2a}{n-4} <p <p_a(n)$. \vskip .1in

For any given $x\in \Omega$, let $0<r<R$ and $B_r(x) \subset B_R(x) \subset \Omega$, we choose
$u \in W^{4,2}_{loc}(\Omega)$ and $|x|^a|u|^{p+1} \in L_{loc}^1(\Omega)$ and define
\begin{align}\label{eq:2.02}
E(r; x,u):= & r^{\frac{4(p+1)+2a}{p-1}-n} \int_{B_r(x)} \dfrac{1}{2} (\Delta u)^2-\dfrac{1}{p+1} |x|^a |u|^{p+1} \nonumber \\[0.06cm]
& +\dfrac{4+a}{2(p-1)} \left (n-2-\dfrac{4+a}{p-1} \right )\dfrac{d}{d r} \left (r^{\frac{8+2a}{p-1}+1-n}\int_{\partial B_r(x)} u^2 \right ) \nonumber \\[0.06cm]
& +\dfrac{4+a}{2(p-1)} \left (n-2-\dfrac{4+a}{p-1} \right )\dfrac{d}{dr} \left (
r^{\frac{8+2a}{p-1}+2-n}\int_{\partial B_r(x)} u^2 \right ) \nonumber \\[0.06cm]
& +\dfrac{r^3}{2} \dfrac{d}{dr} \left [ r^{\frac{8+2a}{p-1}+1-n}\int_{\partial B_r(x)}
\left (\dfrac{4+a}{p-1} r^{-1} u +\dfrac{\partial u}{\partial r} \right )^2\right ] \nonumber \\[0.06cm]
& +\dfrac{1}{2} \dfrac{d}{dr}\left [r^{ \frac{8+2a}{p-1}+4-n}\int_{\partial B_r(x)}\left (|\nabla u|^2-\left |\dfrac{\partial u}{\partial r} \right |^2
\right ) \right ] \nonumber \\[0.06cm]
& +\dfrac{1}{2}r^{ \frac{8+2a}{p-1}+3-n}\int_{\partial B_r(x)}\left (|\nabla u|^2-\left |\dfrac{\partial u}{\partial r} \right |^2
\right ).
\end{align}Then, we can investigate a monotonicity formula.

\begin{theorem}\label{eq:t2.1}
Suppose that $n\ge 5$, $a \ge 0$ and $p > \dfrac{n+4+2a}{n-4}$, $u \in W^{4,2}_{loc} (\Omega)$ and $|x|^a |u|^{p+1} \in
L_{loc}^1(\Omega)$
is a weak solution of (\ref{eq:1.1}). Then $E(r;x,u)$ is non-decreasing in $r \in (0,R)$. Furthermore, we have
\begin{equation}\label{eq:2.03}
\dfrac{d}{dr}E(r;0,u) \ge c(n,p,a) r^{-n+2+\frac{8+2a}{p-1}}\int_{\partial B_r}
\left (\dfrac{4+a}{p-1} r^{-1} u+\dfrac{\partial u}{\partial r}
\right )^2 dS,
\end{equation}where the constant $c(n,p,a) >0$ is only relevant to $n$, $p$ and $a$.
\end{theorem}

\begin{proof} We follow the lines of analysis process in \cite{Davila} to prove the conclusion.
From the variational of the equation (\ref{eq:1.1}), we define the rescaled energy function
\begin{equation}\label{eq:2.04}
\hat{E}(\tau):=\tau^{\frac{4(p+1)+2a}{p-1}-n}\int_{B_{\tau}} \dfrac{1}{2}(\Delta u)^2-\dfrac{1}{p+1} |x|^a |u|^{p+1}.
\end{equation}Denote
\begin{equation*}
v:=\Delta u
\end{equation*}and
\begin{equation*}
u^{\tau}(x):=\tau^{\frac{4+a}{p-1}}u(\tau x),\quad v^{\tau}(x):=\tau^{\frac{4+a}{p-1}+2}v(\tau x).
\end{equation*}It is easy to check that
\begin{equation*}
v^{\tau}=\Delta u^{\tau}\;\ \mbox{and}\;\ \Delta v^{\tau}=|x|^a |u^{\tau}(x)|^{p-1}u(\tau x),
\end{equation*}and taking the derivative of the first equality in $\tau$
to get
\begin{equation*}
\dfrac{dv^{\tau}}{d \tau} = \Delta \dfrac{du^{\tau}}{d\tau }.
\end{equation*}We observe that differentiation in $\tau$ exchanges with differentiation and integration in $x$. Rescaling in (\ref{eq:2.04})
to yield
\begin{equation*}
\hat{E}(\tau)=\int_{B_1} \dfrac{1}{2} (v^{\tau})^2-\dfrac{1}{p+1} |x|^a |u^{\tau}|^{p+1}.
\end{equation*}Differentiating the function $\hat{E}(\tau)$ in $\tau$, we obtain
\begin{align}\label{eq:2.05}
\dfrac{d\hat{E}(\tau)}{d\tau} & =\int_{B_1} v^{\tau} \dfrac{dv^{\tau}}{d \tau}-|x|^a |u^{\tau}|^{p-1}u^{\tau}
\dfrac{du^{\tau}}{d \tau} \nonumber \\[0.06cm]
& = \int_{B_1} v^{\tau}\Delta \dfrac{d u^{\tau}}{d \tau}-\Delta v^{\tau} \dfrac{d u^{\tau}}{d\tau} \nonumber \\[0.06cm]
& =\int_{\partial B_1} v^{\tau} \dfrac{\partial }{\partial r}\dfrac{d u^{\tau}}{d \tau}
-\dfrac{\partial v^{\tau}}{\partial r}\dfrac{d u^{\tau}}{d \tau}.
\end{align}

In the following, all derivations of $u^{\tau}$ in the $r=|x|$ variable will be expressed by the
derivations in the $\tau$ variable.

From the definition of $u^{\tau}$ and $v^{\tau}$, differentiating in $\tau$ implies
\begin{equation}\label{eq:2.06}
\dfrac{d u^{\tau}}{d \tau}(x)
 =\dfrac{1}{\tau} \left [\dfrac{4+a}{p-1} u^{\tau}(x)+r \dfrac{\partial u^{\tau}}{\partial r}(x)\right ]
\end{equation}and
\begin{equation*}
\dfrac{dv^{\tau}}{d \tau}(x) =\dfrac{1}{\tau} \left [\dfrac{2(p+1)+a}{p-1}v^{\tau}(x) +r \dfrac{\partial v^{\tau}}{\partial r}(x) \right ].
\end{equation*}In (\ref{eq:2.06}), differentiating in $\tau$ once again yields
\begin{equation*}
\tau \dfrac{d^2 u^{\tau}}{d \tau^2}+\dfrac{d u^{\tau}}{d \tau}=\dfrac{4+a}{p-1}
\dfrac{du^{\tau}}{d \tau}+r \dfrac{\partial }{\partial
r}\dfrac{d u^{\tau}}{d \tau}.
\end{equation*}So we get
\begin{eqnarray*}
& r\dfrac{\partial }{\partial r} \dfrac{d u^{\tau}}{d \tau}  =\tau \dfrac{d^2 u^{\tau}}{d \tau^2}
+\dfrac{p-5-a}{p-1}\dfrac{du^{\tau}}{d \tau}, &\\[0.06cm]
& r \dfrac{\partial v^{\tau}}{\partial r} =\tau \dfrac{d v^{\tau}}{d \tau}-\dfrac{2(p+1)+a}{p-1} v^{\tau}.&
\end{eqnarray*}Inserting the above two equalities into (\ref{eq:2.05}), we find
\begin{align}\label{eq:2.07}
\dfrac{d \hat{E}}{d \tau}= & \int_{\partial B_1} v^{\tau} \left (\tau \dfrac{d^2 u^{\tau}}{d \tau^2}
+\dfrac{p-5-a}{p-1}\dfrac{d u^{\tau}}{d \tau} \right ) \nonumber \\[0.06cm]
& -\dfrac{d u^{\tau}}{d \tau} \left ( \tau \dfrac{d v^{\tau}}{d \tau}
-\dfrac{2(p+1)+a}{p-1} v^{\tau} \right )\nonumber \\[0.06cm]
= & \int_{\partial B_1} \tau v^{\tau}\dfrac{d^2 u^{\tau}}{d \tau^2}+3v^{\tau} \dfrac{d u^{\tau}}{d \tau}-\tau
\dfrac{du^{\tau}}{d \tau}\dfrac{d v^{\tau}}{d \tau}.
\end{align}

Now, we need to represent the function $v^{\tau}$ by the use of a combination of $u^{\tau}$
and the derivation of $u^{\tau}$ in $\tau$. Taking derivative of (\ref{eq:2.06}) in $r$,
we obtain on $\partial B_1$
\begin{align*}
\dfrac{\partial ^2 u^{\tau}}{\partial r^2} = & \tau \dfrac{\partial }{\partial r}\dfrac{d u^{\tau}}{d \tau}
-\dfrac{p+3+a}{p-1}\dfrac{\partial u^{\tau}}{\partial r} \\[0.06cm]
= & \tau^2 \dfrac{d^2 u^{\tau}}{d \tau^2}+\dfrac{p-5-a}{p-1} \tau \dfrac{d u^{\tau}}{d \tau} \\[0.06cm]
& -\dfrac{p+3+a}{p-1} \left ( \tau \dfrac{d u^{\tau}}{d \tau}-\dfrac{4+a}{p-1} u^{\tau} \right )\\[0.06cm]
= & \tau^2 \dfrac{d^2 u^{\tau}}{d \tau^2}-\dfrac{8+2a}{p-1} \tau \dfrac{du^{\tau}}{d \tau}
+\dfrac{(4+a)(p+3+a)}{(p-1)^2} u^{\tau}.
\end{align*}Using spherical coordinates to write $u^{\tau}(x)=
u^{\tau}(r,\theta)$ with $r=|x|$ and $\theta=\dfrac{x}{|x|} \in \mathbb{S}^{n-1}$, then on $\partial B_1$, we get
\begin{align}\label{eq:2.08}
v^{\tau} = & \dfrac{\partial^2 u^{\tau}}{\partial r^2}+\dfrac{n-1}{r}\dfrac{\partial u^{\tau}}{\partial r}
+\dfrac{1}{r^2} \Delta_{\theta} u^{\tau} \nonumber \\[0.06cm]
= & \tau^2\dfrac{d^2 u^{\tau}}{d \tau^2}+ \left [n-1-\dfrac{8+2a}{p-1} \right ]\tau \dfrac{du^{\tau}}{d \tau} \\[0.06cm]
 & +\dfrac{4+a}{p-1} \left [\dfrac{4+a}{p-1}-n+2 \right ]u^{\tau}+\Delta_{\theta} u^{\tau} \nonumber \\[0.06cm]
= & \tau^2\dfrac{d^2 u^{\tau}}{d \tau^2}+\rho \tau \dfrac{du^{\tau}}{d \tau}+\gamma u^{\tau} +\Delta_{\theta}u^{\tau},
\end{align}where $\Delta_{\theta}$ is the Laplace-Beltrami operator on $\partial B_1$ and $\nabla_{\theta}$ (see section 3) is the tangential
derivative on $\partial B_1$, and
\begin{equation*}
\rho : =n-1-\dfrac{8+2a}{p-1},\quad \gamma : =\dfrac{4+a}{p-1}\left ( \dfrac{4+a}{p-1}-n+2 \right ).
\end{equation*}

Substituting (\ref{eq:2.08}) into (\ref{eq:2.07}), we have
\begin{align}\label{eq:2.09}
\dfrac{d}{d\tau} \hat{E}(\tau)= & \int_{\partial B_1} \tau \left (\tau^2\dfrac{d^2 u^{\tau}}{d \tau^2}
+\rho \tau \dfrac{du^{\tau}}{d \tau}+\gamma u^{\tau} \right )\dfrac{d^2 u^{\tau}}{d\tau^2} \nonumber \\[0.06cm]
& +3 \left (\tau^2\dfrac{d^2 u^{\tau}}{d \tau^2}+\rho \tau \dfrac{du^{\tau}}{d \tau}+\gamma u^{\tau} \right )
\dfrac{du^{\tau}}{d \tau} \nonumber \\[0.06cm]
& -\tau \dfrac{d u^{\tau}}{d \tau} \dfrac{d }{d\tau} \left (\tau^2\dfrac{d^2 u^{\tau}}{d \tau^2}
+\rho \tau \dfrac{du^{\tau}}{d \tau}+\gamma u^{\tau} \right )\nonumber \\[0.06cm]
& +\int_{\partial B_1}\tau \Delta_{\theta} u^{\tau}\dfrac{d^2 u^{\tau}}{d \tau^2}
+3\Delta_{\theta} u^{\tau}\dfrac{d u^{\tau}}{d \tau}-\tau \dfrac{d u^{\tau}}{d \tau} \Delta_{\theta}\dfrac{d u^{\tau}}{d \tau} \nonumber \\[0.06cm]
=: & T_1+T_2.
\end{align}The calculation for $T_1$ is processed as follows
\begin{align}\label{eq:2.10}
T_1 = & \int_{\partial B_1} \tau \left (\tau^2\dfrac{d^2 u^{\tau}}{d \tau^2}
+\rho \tau \dfrac{du^{\tau}}{d \tau}+\gamma u^{\tau} \right )\dfrac{d^2 u^{\tau}}{d\tau^2} \nonumber \\[0.06cm]
& +3 \left (\tau^2\dfrac{d^2 u^{\tau}}{d \tau^2}+\rho \tau \dfrac{du^{\tau}}{d \tau}+\gamma u^{\tau} \right )
\dfrac{du^{\tau}}{d \tau} \nonumber \\[0.06cm]
& -\tau \dfrac{d u^{\tau}}{d \tau} \dfrac{d }{d\tau} \left (\tau^2\dfrac{d^2 u^{\tau}}{d \tau^2}
+\rho \tau \dfrac{du^{\tau}}{d \tau}+\gamma u^{\tau} \right )\nonumber \\[0.06cm]
= & \int_{\partial B_1} \tau^3 \left (\dfrac{d^2 u^{\tau}}{d\tau^2} \right )^2 +\tau^2 \dfrac{d^2 u^{\tau}}{d \tau^2}
\dfrac{d u^{\tau}}{d \tau} +\gamma \tau u^{\tau} \dfrac{d^2 u^{\tau}}{d \tau^2} +3\gamma u^{\tau} \dfrac{d u^{\tau}}{d \tau} \nonumber \\[0.06cm]
& +(2\rho -\gamma)\tau \left (\dfrac{d u^{\tau}}{d \tau} \right )^2
-\tau^3 \dfrac{d u^{\tau}}{d \tau}\dfrac{d^3 u^{\tau}}{d\tau^3} \nonumber \\[0.06cm]
= & \int_{\partial B_1} 2\tau^3 \left (\dfrac{d^2 u^{\tau}}{d\tau^2} \right )^2 +4\tau^2 \dfrac{d^2 u^{\tau}}{d \tau^2}
\dfrac{d u^{\tau}}{d \tau} +2(\rho -\gamma)\tau \left (\dfrac{d u^{\tau}}{d \tau} \right )^2\nonumber \\[0.06cm]
& +\dfrac{\gamma}{2} \dfrac{d^2}{d \tau^2} \left [\tau \left ( u^{\tau}\right )^2 \right ] -\dfrac{1}{2} \dfrac{d }{d\tau} \left [ \tau^3\dfrac{d }{d\tau}
\left (\dfrac{du^{\tau}}{d\tau} \right )^2 \right ] +\dfrac{\gamma}{2}\dfrac{d (u^{\tau})^2}{d \tau} \nonumber \\[0.06cm]
\ge & \int_{\partial B_1} \dfrac{\gamma}{2} \dfrac{d^2}{d \tau^2} \left [\tau \left ( u^{\tau}\right )^2 \right ]
-\dfrac{1}{2} \dfrac{d }{d\tau} \left [ \tau^3\dfrac{d }{d\tau}
\left (\dfrac{du^{\tau}}{d\tau} \right )^2 \right ] +\dfrac{\gamma}{2}\dfrac{d (u^{\tau})^2}{d \tau}.
\end{align}Here choosing $p>\dfrac{n+4+2a}{n-4}$, it implies that
\begin{equation*}
\rho -\gamma =\left (n-1-\dfrac{8+2a}{p-1} \right )-\dfrac{4+a}{p-1}\left (\dfrac{4+a}{p-1}-n+2 \right )>1
\end{equation*}and
\begin{align}\label{eq:2.11}
& 2\tau^3 \left (\dfrac{d^2 u^{\tau}}{d\tau^2} \right )^2 +4\tau^2 \dfrac{d^2 u^{\tau}}{d \tau^2}
\dfrac{d u^{\tau}}{d \tau} +2(\rho -\gamma)\tau \left (\dfrac{d u^{\tau}}{d \tau} \right )^2 \nonumber \\[0.06cm]
&\; =2\tau \left ( \tau \dfrac{d^2 u^{\tau}}{d\tau^2}+ \dfrac{d u^{\tau}}{d \tau}\right )^2 +2(\rho-\gamma-1)
\tau \left (\dfrac{d u^{\tau}}{d \tau} \right )^2 \nonumber \\[0.06cm]
& \; \ge 0.
\end{align}Integrating by parts on $ \partial B_1$, we get
\begin{align}\label{eq:2.12}
T_2 = & \int_{\partial B_1}-\tau \nabla_{\theta} u^{\tau} \nabla_{\theta} \dfrac{d^2 u^{\tau}}{d \tau^2}
-3 \nabla_{\theta} u^{\tau} \nabla_{\theta} \dfrac{d u^{\tau}}{d \tau}+\tau \left | \nabla_{\theta} \dfrac{d u^{\tau}}{d \tau} \right |^2 \nonumber \\[0.06cm]
=& -\dfrac{\tau}{2}\dfrac{d^2}{d\tau^2} \int_{\partial B_1} |\nabla_{\theta} u^{\tau}|^2-\dfrac{3}{2}\dfrac{d }{d\tau} \int_{\partial B_1}
\left |\nabla_{\theta} u^{\tau} \right |^2+2\tau \int_{\partial B_1}\left |\nabla_{\theta} \dfrac{d u^{\tau}}{d\tau} \right |^2 \nonumber \\[0.06cm]
= & -\dfrac{1}{2} \dfrac{d^2}{d\tau^2} \left (\tau \int_{\partial B_1} \left |\nabla_{\theta} u^{\tau} \right |^2 \right )
-\dfrac{1}{2} \dfrac{d }{d \tau} \int_{\partial B_1} \left |\nabla_{\theta} u^{\tau} \right |^2
+2\tau \int_{\partial B_1}\left |\nabla_{\theta} \dfrac{d u^{\tau}}{d\tau} \right |^2 \nonumber \\[0.06cm]
\ge & -\dfrac{1}{2} \dfrac{d^2}{d\tau^2} \left (\tau \int_{\partial B_1} \left |\nabla_{\theta} u^{\tau} \right |^2 \right )
-\dfrac{1}{2} \dfrac{d }{d \tau} \int_{\partial B_1} \left |\nabla_{\theta} u^{\tau} \right |^2.
\end{align}

On the other hand, we observe that all terms in (\ref{eq:2.10}) and (\ref{eq:2.12}) by the use of the rescaling can be expressed as follows:
\begin{align*}
\int_{\partial B_1}\dfrac{d }{d \tau}(u^{\tau})^2 & =\dfrac{d}{d \tau} \left (\tau^{\frac{8+2a}{p-1}+1-n}\int_{\partial B_{\tau}} u^2 \right ),\\[0.06cm]
\int_{\partial B_1} \dfrac{d^2 }{d \tau^2} \left [\tau (u^{\tau})^2 \right ] & =\dfrac{d^2}{d\tau^2} \left (
\tau^{\frac{8+2a}{p-1}+2-n}\int_{\partial B_{\tau}} u^2 \right ),\\[0.06cm]
\int_{\partial B_1} \dfrac{d}{d \tau}\left [\tau^3 \dfrac{d }{d \tau}\left ( \dfrac{d u^{\tau}}{d \tau}\right )^2 \right ] &=
\dfrac{d}{d \tau} \left [\tau^3 \dfrac{d}{d\tau} \left ( \tau^{\frac{8+2a}{p-1}+1-n}\int_{\partial B_{\tau}}
\left (\dfrac{4+a}{p-1} \tau^{-1} u +\dfrac{\partial u}{\partial r} \right )^2\right ) \right ],\\[0.06cm]
\dfrac{d^2}{d \tau^2} \left (\tau \int_{\partial B_1}\left | \nabla_{\theta}u^{\tau} \right |^2 \right )
& =\dfrac{d^2}{d\tau^2} \left [\tau^{ \frac{8+2a}{p-1}+4-n}\int_{\partial B_{\tau}}\left (|\nabla u|^2-\left |\dfrac{\partial u}{\partial r} \right |^2
\right ) \right ],\\[0.06cm]
\dfrac{d}{d \tau} \left ( \int_{\partial B_1}\left | \nabla_{\theta}u^{\tau} \right |^2 \right )
& =\dfrac{d}{d\tau} \left [\tau^{ \frac{8+2a}{p-1}+3-n}\int_{\partial B_{\tau}}\left (|\nabla u|^2-\left |\dfrac{\partial u}{\partial r} \right |^2
\right ) \right ].
\end{align*}Combining with (\ref{eq:2.09})-(\ref{eq:2.12}), we obtain
\begin{align*}
\dfrac{d \hat{E}(\tau)}{d \tau} \ge & \ \dfrac{\gamma}{2} \dfrac{d}{d \tau} \left (\tau^{\frac{8+2a}{p-1}+1-n}\int_{\partial B_{\tau}} u^2 \right ) +\dfrac{\gamma}{2} \dfrac{d^2}{d\tau^2} \left (
\tau^{\frac{8+2a}{p-1}+2-n}\int_{\partial B_{\tau}} u^2 \right ) \\[0.06cm]
& -\dfrac{1}{2} \dfrac{d}{d \tau} \left [\tau^3 \dfrac{d}{d\tau} \left ( \tau^{\frac{8+2a}{p-1}+1-n}\int_{\partial B_{\tau}}
\left (\dfrac{4+a}{p-1} \tau^{-1} u +\dfrac{\partial u}{\partial r} \right )^2\right ) \right ] \\[0.06cm]
& -\dfrac{1}{2} \dfrac{d}{d\tau} \left [\tau^{ \frac{8+2a}{p-1}+3-n}\int_{\partial B_{\tau}}\left (|\nabla u|^2-\left |\dfrac{\partial u}{\partial r} \right |^2
\right ) \right ] \\[0.06cm]
& -\dfrac{1}{2} \dfrac{d^2}{d\tau^2} \left [\tau^{ \frac{8+2a}{p-1}+4-n}\int_{\partial B_{\tau}}\left (|\nabla u|^2-\left |\dfrac{\partial u}{\partial r} \right |^2
\right ) \right ].
\end{align*}Therefore, combining the inequality with (\ref{eq:2.02}) and (\ref{eq:2.04}), we get the inequality (\ref{eq:2.03}).
From the properties of integration, we conclude that $E(r;x,u)$ is nonindecreasing in $r \in (0,R)$.
\end{proof}

\begin{remark}
From (\ref{eq:2.10})-(\ref{eq:2.12}), it implies that we can take
\begin{equation*}
c(n,p,a)=2(\rho -\gamma -1).
\end{equation*}In particular, if $p=\dfrac{n+4+2a}{n-4}$, then $c(n,p,a)=\dfrac{n^2-4n+8}{2}>0$. Therefore, Theorem \ref{eq:t2.1} also holds
for $p=\dfrac{n+4+2a}{n-4}$.
\end{remark}

\begin{corollary}\label{eq:c2.1}
If the hypotheses of Theorem \ref{eq:t2.1} hold and $E(\sigma;0,u) \equiv \mbox{const}$, for all $\sigma \in (0, R)$,
then $u$ is homogeneous in $B_R \backslash \{0\}$, i.e.,
\begin{equation}\label{eq:2.13}
u(\tau x)=\tau^{-\frac{4+a}{p-1}}u(x),\;\forall \tau \in (0,1],\;\ \ x \in B_R \backslash \{0\}.
\end{equation}
\end{corollary}

\begin{proof}
Taking arbitrarily $r_1, r_2 \in (0,R)$ with $r_1< r_2$, we obtain from Theorem \ref{eq:t2.1} that
\begin{align*}
0 & = E(r_2; 0, u)-E(r_1;0,u) \\
& =\int_{r_1}^{r_2} \dfrac{d }{d \sigma}E(\sigma; 0,u) d \sigma \\
& \ge c(n,p,a) \int_{B_{r_2} \backslash B_{r_1}}
\dfrac{\left (\frac{4+a}{p-1}\sigma^{-1} u+\frac{\partial u}{\partial \sigma}
\right )^2}{|x|^{n-2-\frac{8+2a}{p-1}}}dx.
\end{align*}This implies
\begin{equation*}
\dfrac{4+a}{p-1} \sigma^{-1} u+\dfrac{\partial u}{\partial \sigma}=0,\;\mbox{a.e.}\;\mbox{in}\;\ B_{R}\backslash \{0\}.
\end{equation*}Hence for any fixed $x \in B_R \backslash \{0\}$, we have
\begin{equation*}
\dfrac{d}{d \tau}\left (\tau^{\frac{4+a}{p-1}} u(\tau x) \right ) \equiv 0, \forall \tau \in (0,1].
\end{equation*}Obviously, the equality (\ref{eq:2.13}) holds.
\end{proof}\vskip .1in

The following basic integral estimates for solutions (whether positive or sign-changing)
of (\ref{eq:1.1}) follows from the rescaled test function method.

\begin{lemma}(\cite[Lemma 2.3]{Wei-Ye})\label{eq:l2.1}
For any $\zeta \in C^4(\mathbb{R}^n)$ and $\eta \in C^{\infty}_0 (\mathbb{R}^n)$, we have the following identities
\begin{align}
\int_{\mathbb{R}^n} (\Delta^2 \zeta) \zeta \eta^2 dx= & \int_{\mathbb{R}^n} [\Delta (\zeta \eta)]^2dx +\int_{\mathbb{R}^n}
\left [-4(\nabla \zeta \cdot \nabla \eta)^2+2 \zeta \Delta \zeta | \nabla \eta|^2 \right ]dx \nonumber \\[0.06cm]
& \label{eq:2.14} +\int_{\mathbb{R}^n}\zeta^2 \left [2\nabla (\Delta \eta)\cdot \nabla \eta+(\Delta \eta)^2 \right ]dx
\end{align}and
\begin{equation}
\label{eq:2.15} 2 \int_{\mathbb{R}^n} |\nabla \zeta|^2 |\nabla \eta |^2 dx = \int_{\mathbb{R}^n}\left [ 2 \zeta
(-\Delta \zeta) |\nabla \eta|^2+\zeta^2 \Delta (|\nabla \eta|^2)\right ] dx.
\end{equation}
\end{lemma}\vskip .1in

\begin{lemma}\label{eq:l2.2}
Let $u\in C^4(\mathbb{R}^n)$ be a stable solution of (\ref{eq:1.1}). Then for large enough
$m$, we get that for all $\psi \in C^4_0(\mathbb{R}^n)$ with $0 \le \psi \le 1$
\begin{equation*}
\int_{\mathbb{R}^n} \left ( |\Delta u|^2+|x|^a |u|^{p+1} \right )\psi^{2m} \le C \int_{\mathbb{R}^n} |x|^{-\frac{2a}{p-1}}|G(\psi^m)|^{\frac{p+1}{p-1}},
\end{equation*}where $G(\psi^m)=|\nabla \psi|^4+\psi^{2(2-m)}\Big [|\nabla (\Delta \psi^m)\cdot \nabla \psi^m|+|\Delta \psi^m|^2
+|\Delta |\nabla \psi^m|^2| \Big ]$.

Furthermore, we find
\begin{equation}\label{eq:2.16}
\int_{B_R(x)} (|\Delta u|^2+|x|^a |u|^{p+1}) \le CR^{-4} \int_{B_{2R}(x)\backslash B_{R}(x)}u^2+ CR^{-2}
\int_{B_{2R}(x)\backslash B_{R}(x)} |u \Delta u|
\end{equation}and
\begin{equation}\label{eq:2.17}
\int_{B_R(x)}(|\Delta u|^2+|x|^a |u|^{p+1}) \le CR^{n-\frac{4(p+1)+2a}{p-1}}
\end{equation}for all $B_R(x)$.
Here the constant $C$ does not depend on $R$ and $u$.
\end{lemma}

\begin{proof}
Since $u$ is a stable solution of (\ref{eq:1.1}), we choose arbitrarily $\zeta \in C^4_0(\mathbb{R}^n)$ and find
\begin{equation}\label{eq:2.18}
\int_{\mathbb{R}^n} |x|^a |u|^{p-1}u \zeta =\int_{\mathbb{R}^n} \Delta u \Delta \zeta
\end{equation}and
\begin{equation}\label{eq:2.19}
p\int_{\mathbb{R}^n} |x|^a |u|^{p-1} \zeta^2 \le \int_{\mathbb{R}^n} |\Delta \zeta|^2.
\end{equation}Testing (\ref{eq:2.18}) on $\zeta=u \psi^2$ for $\psi \in C^4_0(\mathbb{R}^n)$, we obtain
\begin{equation*}
\int_{\mathbb{R}^n}|x|^a |u|^{p+1} \psi^2 =\int_{\mathbb{R}^n} \Delta u \Delta (u\psi^2).
\end{equation*}Testing (\ref{eq:2.19}) on $\zeta =u \psi$ to yield
\begin{equation*}
p \int_{\mathbb{R}^n} |x|^a |u|^{p+1} \psi^2 \le \int_{\mathbb{R}^n} [\Delta (u\psi)]^2.
\end{equation*}Combining the above two results with (\ref{eq:2.14}), we get
\begin{align*}
(p-1)\int_{\mathbb{R}^n} |x|^a |u|^{p+1} \psi^2 \le & \int_{\mathbb{R}^n} \left [ 4(\nabla u\cdot \nabla \psi)^2-2u\Delta u
|\nabla \psi|^2 \right ]dx \\[0.06cm]
& +\int_{\mathbb{R}^n} u^2 \left [ 2 |\nabla (\Delta \psi)\cdot \nabla \psi|+|\Delta \psi|^2 \right ]dx.
\end{align*}Again a direct application of the identity (\ref{eq:2.15}) leads to
\begin{align}\label{eq:2.20}
& \int_{\mathbb{R}^n} |x|^a |u|^{p+1} \psi^2 dx \le C \int_{\mathbb{R}^n} |uv| |\nabla \psi|^2 dx \nonumber \\[0.06cm]
& +C\int_{\mathbb{R}^n}
u^2 \left [|\nabla (\Delta \psi) \cdot \nabla \psi|+ |\Delta \psi|^2 +|\Delta (|\nabla \psi|^2)| \right ]dx.
\end{align}Since $\Delta (u\psi)=v\psi+2\nabla u\cdot \nabla \psi+ u \Delta \psi$, we obtain
\begin{align}\label{eq:2.21}
& \int_{\mathbb{R}^n} v^2\psi^2 dx \le C \int_{\mathbb{R}^n} |uv| |\nabla \psi |^2 dx \nonumber \\[0.06cm]
& +C \int_{\mathbb{R}^n} u^2 \left [ |\nabla (\Delta
\psi)\cdot \nabla \psi| +|\Delta \psi |^2+|\Delta |\nabla \psi|^2| \right ]dx.
\end{align}We take a cut-off function $\psi \in C^{\infty}_0 (B_{2R}(x))$ such that $\psi \equiv 1$ in $B_R(x)$ and
for $k \le 3$, $|\nabla^k \psi|\le \dfrac{C}{R^k}$. Combining (\ref{eq:2.20}) with (\ref{eq:2.21}), we get
\begin{equation*}
\int_{B_R(x)}(v^2 + |x|^a|u|^{p+1})dx \le CR^{-4} \int_{B_{2R}(x)\backslash B_R(x)} u^2 +CR^{-2} \int_{B_{2R}(x) \backslash B_R(x)}
|uv|.
\end{equation*}Noting that the constant $C$ does not depend on $R$ and $u$.

Next, the functions $\psi$ in (\ref{eq:2.20}) and (\ref{eq:2.21}) are replaced by $\psi^m$,
where $m$ is a large integer. Then
\begin{align*}
\int_{\mathbb{R}^n} & \left [ |x|^a  |u|^{p+1} +v^2 \right ] \psi^{2m} dx \le C \int_{\mathbb{R}^n} |uv| \psi^{2(m-1)} |\nabla \psi|^2\\
& +C\int_{\mathbb{R}^n} u^2 \left [|\nabla (\Delta \psi^m)\cdot \nabla \psi^m|+|\Delta \psi^m|^2+|\Delta |\nabla \psi^m|^2|\right ]dx.
\end{align*}A simple application of Young's inequality yields
\begin{equation*}
\int_{\mathbb{R}^n}|uv|\psi^{2(m-1)}|\nabla \psi|^2 \le \dfrac{1}{2C} \int_{\mathbb{R}^n} v^2 \psi^{2m} +C
\int_{\mathbb{R}^n} u^2 \psi^{2(m-2)}|\nabla \psi|^4.
\end{equation*}Therefore we get
\begin{equation}\label{eq:2.22}
\int_{\mathbb{R}^n}|x|^a |u|^{p+1} \psi^{2m}+v^2 \psi^{2m}dx \le C\int_{\mathbb{R}^n} u^2 \psi^{2(m-2)} G(\psi^m),
\end{equation}where $G(\psi^m)=|\nabla \psi|^4+\psi^{2(2-m)}\left [|\nabla (\Delta \psi^m)\cdot \nabla \psi^m|+|\Delta \psi^m|^2
+|\Delta |\nabla \psi^m|^2| \right ]$. Utilizing H\"{o}lder's inequality to find
\begin{align*}
\int_{\mathbb{R}^n} u^2 \psi^{2(m-2)} & G(\psi^m) =\int_{\mathbb{R}^n}|x|^{\frac{2a}{p+1}}u^2\psi^{2(m-2)} |x|^{-\frac{2a}{p+1}}G(\psi^m) \\
& \le \left (\int_{\mathbb{R}^n}|x|^a|u|^{p+1} \psi^{(m-2)(p+1)} \right )^{\frac{2}{p+1}}
\left (\int_{\mathbb{R}^n}|x|^{-\frac{2a}{p-1}}G(\psi^m)^{\frac{p+1}{p-1}} \right )^{\frac{p-1}{p+1}}.
\end{align*}Choosing $m$ large enough such that $(m-2)(p+1) \ge 2m$, and combining with (\ref{eq:2.22}), we get
\begin{equation*}
\int_{\mathbb{R}^n} (|\Delta u|^2 + |x|^a|u|^{p+1}) \psi^{2m}dx \le
C\int_{\mathbb{R}^n} |x|^{-\frac{2a}{p+1}} G(\psi^m)^{\frac{p+1}{p-1}},
\end{equation*}where the constant $C$ only depends on $n$, $p$, $a$, $m$ and $\psi$. In the above inequality, we take a cut-off function $\psi \in C^{\infty}_0 (B_{2R}(x))$ such that $\psi \equiv 1$ in $B_R(x)$ and $|\nabla ^i \psi| \le \dfrac{C}{R^i}$ for $i=1,2,3$ once again. Then
\begin{align*}
\int_{B_R(x)} (|\Delta u|^2 + |x|^a|u|^{p+1}) dx & \le C \int_{B_{2R}(x) \backslash B_R(x)} |x|^{-\frac{2a}{p-1}}R^{-4\frac{p+1}{p-1}}dx \\[0.06cm]
& \le CR^{n-\frac{4(p+1)+2a}{p-1}}.
\end{align*}The proof is completed.
\end{proof}

\section{Proof of Theorem \ref{eq:t1.1}.}

We first obtain a nonexistence result for homogeneous stable solution of (\ref{eq:1.1}).

\begin{theorem}\label{eq:t3.1}
For any $p \in \left (\dfrac{n+4+2a}{n-4},p_a(n) \right )$, assume
that $u \in W^{2,2}_{loc}\left (\mathbb{R}^n \backslash \{0\}\right )$ is a homogeneous, stable solution of (\ref{eq:1.1}), 
and $|x|^a|u|^{p+1} \in L^1_{loc}
(\mathbb{R}^n \backslash \{0\})$, where $p_a(n)$ is given by (\ref{eq:2.01}). Then $u \equiv 0$.
\end{theorem}

\begin{proof}
From the conditions of Theorem,
we can assume that there exists a $w\in W^{2,2}(\mathbb{S}^{n-1})$ such that in polar coordinates
\begin{equation*}
u(r,\theta)=r^{-\frac{4+a}{p-1}} w(\theta).
\end{equation*}Since $u\in W^{2,2}(B_2 \backslash B_1)$ and $|x|^a|u|^{p+1} \in L^1(B_2 \backslash B_1)$, it implies
that $w \in W^{2,2}(\mathbb{S}^{n-1})\cap
L^{p+1}(\mathbb{S}^{n-1})$. A straightforward calculation of (\ref{eq:1.1}) to get
\begin{equation}\label{eq:3.1}
\Delta^2 _{\theta} w-\ell_1 \Delta_{\theta}w+\ell_2 w=|w|^{p-1}w,
\end{equation}where
\begin{align*}
\ell_1  =\left (\dfrac{4+a}{p-1}+2 \right )\left (n-4-\dfrac{4+a}{p-1} \right )+\dfrac{4+a}{p-1} \left (n-2-\dfrac{4+a}{p-1} \right ),\\[0.09cm]
\ell_2  = \dfrac{4+a}{p-1}\left (\dfrac{4+a}{p-1}+2 \right ) \left (n-4-\dfrac{4+a}{p-1}
\right ) \left (n-2-\dfrac{4+a}{p-1} \right ).
\end{align*}From $w \in W^{2,2}(\mathbb{S}^{n-1})$, multiplying (\ref{eq:3.1}) by $w$ and integrating by parts imply
\begin{equation}\label{eq:3.2}
\int_{\mathbb{S}^{n-1}}|\Delta_{\theta} w|^2 +\ell_1|\nabla_{\theta} w|^2 +\ell_2 w^2 =\int_{\mathbb{S}^{n-1}} |w|^{p+1}.
\end{equation}\vskip .05in

On the other hand, for any $\epsilon >0$, we choose an $\zeta_{\epsilon} \in C^{\infty}_0 \left (\left (\frac{\epsilon}{2},\frac{2}{\epsilon} \right ) \right )$
such that $\zeta_{\epsilon} \equiv 1$ in $\left (\epsilon, \frac{1}{\epsilon} \right )$ and
\begin{equation*}
r |\zeta_{\epsilon}'(r)|+r^2 |\zeta_{\epsilon}''(r)| \le C
\end{equation*}for all $r>0$. Since $u$ is a stable solution, we can choose a test function $r^{-\frac{n-4}{2}}w(\theta) \zeta_{\epsilon}(r)$ and get
\begin{equation*}
p \int_{\mathbb{R}^n} |x|^a |u|^{p-1} \left (r^{-\frac{n-4}{2}}w(\theta) \zeta_{\epsilon}(r)
\right )^2 dx \le \int_{\mathbb{R}^n} \left | \Delta \left ( r^{-\frac{n-4}{2}}w(\theta) \zeta_{\epsilon}(r)
\right ) \right |^2 dx.
\end{equation*}A simple calculation implies
\begin{align*}
\Delta \left (r^{-\frac{n-4}{2}}w(\theta)\zeta_{\epsilon}(r) \right ) =& -\dfrac{n(n-4)}{4}r^{-\frac{n}{2}}\zeta_{\epsilon}(r) w(\theta)+r^{-\frac{n}{2}} \zeta_{\epsilon}(r)
\Delta_{\theta} w(\theta)\\
& +3r^{-\frac{n}{2}+1}\zeta_{\epsilon}'(r) w(\theta)+r^{-\frac{n}{2}+2}\zeta_{\epsilon}''(r)w(\theta).
\end{align*}and
\begin{align}\label{eq:3.3}
\begin{split}
p & \int_0^{\infty} \int_{\mathbb{S}^{n-1}} r^a |u|^{p-1} \left (r^{-\frac{n-4}{2}}w(\theta) \zeta_{\epsilon}(r) \right )^2 r^{n-1}dr d\theta \\[0.03cm]
= & p \left ( \int_{\mathbb{S}^{n-1}} |w|^{p+1}d \theta \right ) \left ( \int_0^{\infty} r^{-1}\zeta_{\epsilon}^2(r) dr \right ) \\[0.03cm]
\le & \left (\int_{\mathbb{S}^{n-1}} \Big (|\Delta_{\theta} w|^2+\dfrac{n(n-4)}{2} |\nabla_{\theta} w|^2+
\dfrac{n^2(n-4)^2}{16} w^2 \Big )d \theta \right )\left (\int_0^{\infty} r^{-1}\zeta_{\epsilon}^2(r) dr \right ) \\[0.03cm]
& +O \Big \{ \left ( \int_0^{\infty} \left [ r |\zeta_{\epsilon}'(r)|^2+r^3\zeta_{\epsilon}''(r)^2
+|\zeta_{\epsilon}'(r)|\zeta_{\epsilon}(r)+r\zeta_{\epsilon}(r)|\zeta_{\epsilon}''(r)| \right ]dr \right ) \\[0.03cm]
& \times \int_{\mathbb{S}^{n-1}} \left [ w(\theta)^2+|\nabla_{\theta} w(\theta)|^2 \right ] d\theta  \Big \}.
\end{split}
\end{align}From the definition of $\zeta_{\epsilon}$, one can easily estimate that
\begin{equation*}
\int_0^{\infty} r^{-1}\zeta_{\epsilon}^2(r)dr \ge \int_{\epsilon}^{\frac{1}{\epsilon}}r^{-1}dr \ge |\ln \epsilon|
\end{equation*}and
\begin{equation*}
\int_0^{\infty} \Big [ r \zeta_{\epsilon}'(r)^2+r^3 \zeta_{\epsilon}''(r)^2
+|\zeta_{\epsilon}'(r)|\zeta_{\epsilon}(r)+r\zeta_{\epsilon}(r)|\zeta_{\epsilon}''(r)|\Big ] dr\le C.
\end{equation*}Letting $\epsilon \to 0$, it implies from (\ref{eq:3.3}) that
\begin{equation}\label{eq:3.4}
p\int_{\mathbb{S}^{n-1}} |w|^{p+1} d \theta \le \int_{\mathbb{S}^{n-1}} \left (|\Delta_{\theta} w|^2+\dfrac{n(n-4)}{2} |\nabla_{\theta} w|^2+
\dfrac{n^2(n-4)^2}{16} w^2 \right )d \theta.
\end{equation}Now, combining (\ref{eq:3.2}) with (\ref{eq:3.4}), we obtain
\begin{equation*}
\int_{\mathbb{S}^{n-1}} (p-1)|\Delta_{\theta} w|^2 +\left (p\ell_1-\dfrac{n(n-4)}{2} \right ) |\nabla_{\theta} w|^2+
\left (p\ell_2-\dfrac{n^2(n-4)^2}{16} \right ) w^2 \le 0.
\end{equation*}Since $\dfrac{n+4+2a}{n-4}< p <p_a(n)$, we get from the definition of $p_a(n)$ that
\begin{equation*}
p\ell_1-\dfrac{n(n-4)}{2}>0\;\ \mbox{and}\;\ p\ell_2-
\dfrac{n^2(n-4)^2}{16} >0.
\end{equation*}Therefore we have
\begin{equation*}
w \equiv 0.
\end{equation*} Thus $u\equiv 0$.
\end{proof}

\begin{remark}
One can easily check that
\begin{equation*}
u_s(r)=\ell_2^{\frac{1}{p-1}}r^{-\frac{4+a}{p-1}}
\end{equation*}is a singular solution of (\ref{eq:1.1}) in $\mathbb{R}^n \backslash \{0\}$, where
\begin{equation*}
\beta =\dfrac{4+a}{p-1},\quad \ell_2=\beta(\beta+2)(\beta+4-n)(\beta+2-n).
\end{equation*}

Using the well-known Hardy-Rellich inequality \cite{Rellich} with the best constant
\begin{equation*}
\int_{\mathbb{R}^n} |\Delta \psi|^2 dx \ge \dfrac{n^2(n-4)^2}{16}\int_{\mathbb{R}^n}\dfrac{\psi^2}{|x|^4}dx,\quad \forall \psi \in H^2(\mathbb{R}^n),
\end{equation*}we conclude that the singular solution $u_s$ is stable in $\mathbb{R}^n\backslash \{0\}$ if and only if
\begin{equation*}
p\ell_2 \le \dfrac{n^2(n-4)^2}{16}.
\end{equation*}
\end{remark}\vskip .1in

In what follows, we assume that $u$ is a smooth stable solution of (\ref{eq:1.1}) in $\mathbb{R}^n$ and $\dfrac{n+4+2a}{n-4}
<p< p_a(n)$. Then we obtain the following three lemmas which play an important role in dealing with the supercritical case.\vskip .1in

For all $\tau >0$, we define {\it blowing down} sequences
\begin{equation*}
u^{\tau}(x):=\tau^{\frac{4+a}{p-1}}u(\tau x),\quad v^{\tau}(x):=\tau^{\frac{4+a}{p-1}+2} v(\tau x).
\end{equation*}It is easy to check that $u^{\tau}$ is also a smooth
stable solution of (\ref{eq:1.1}) and for all ball $B_r(x) \subset \mathbb{R}^n$, the following estimate holds
\begin{align*}
& \int_{B_r(x)} \left [ (v^{\tau})^2+|x|^a |u^{\tau}|^{p+1} \right ] dx\\
& = \int_{B_r(x)} \left [ \tau^{\frac{8+2a}{p-1}+4}v(\tau x)^2+|x|^a \tau^{\frac{4+a}{p-1}(p+1)} |u(\tau x)|^{p+1} \right ] dx \\
& = \tau^{\frac{4(p+1)+2a}{p-1}-n} \int_{B_{\tau r}(x)} \left [ v(x)^2+|x|^a |u(x)|^{p+1} \right ] dx \\
& \le C \tau^{\frac{4(p+1)+2a}{p-1}-n}(\tau r)^{n-\frac{4(p+1)+2a}{p-1}}\qquad \mbox{by}\; \mbox{(\ref{eq:2.17})}\\
& = Cr^{n-\frac{4(p+1)+2a}{p-1}}.
\end{align*}Moreover, using H\"{o}lder's inequality to lead to
\begin{align*}
\int_{B_r(x)} (u^{\tau})^2 dx \le & \left (\int_{B_r(x)} |x|^a |(u^{\tau})|^{p+1} dx\right )^{\frac{2}{p+1}}
\left (\int_{B_r(x)}\left (|x|^{-\frac{2a}{p+1}} \right )^{\frac{p+1}{p-1}}dx \right )^{\frac{p-1}{p+1}}\\[0.06cm]
\le & Cr^{n-\frac{8+2a}{p-1}}.
\end{align*}

We note that $u^{\tau}$ are uniformly bounded in $L^{p+1}_{loc}(\mathbb{R}^n)$. From elliptic regularity theory,
it implies that $u^{\tau}$ are also uniformly bounded in $W^{2,2}_{loc}(\mathbb{R}^n)$. Hence,
we can suppose that $u^{\tau} \to u^{\infty}$ weakly
in $W^{2,2}_{loc}(\mathbb{R}^n) \cap L^{p+1}_{loc}(\mathbb{R}^n)$ (if necessary, we can extract a subsequence).
Utilizing standard embeddings, we get
$u^{\tau} \to u^{\infty}$ strongly in $W^{1,2}_{loc}(\mathbb{R}^n)$. Then for any ball $B_R(0)$, applying interpolation
between $L^q$ spaces and noting the above two inequalities, for any $q\in (1,p+1)$, we get
\begin{equation}\label{eq:4.1}
\|u^{\tau}-u^{\infty}\|_{L^q(B_R(0))}\le \|u^{\tau}-u^{\infty}\|_{L^1(B_R(0))}^{\mu} \|u^{\tau}-u^{\infty}\|_{L^{p+1}(B_R(0))}^{1-\mu} \to 0,
\end{equation}as $\tau \to +\infty$, where $\mu \in (0,1)$ satisfying $\dfrac{1}{q}=\mu +\dfrac{1-\mu}{p+1}$. That is, $u^{\tau} \to u^{\infty}$
in $L^q_{loc}(\mathbb{R}^n)$ for any $q \in (1,p+1)$.

Since $u^{\tau}$ is a smooth stable solution of (\ref{eq:1.1}), we get that for any $\zeta \in C_0^{\infty} (\mathbb{R}^n)$
\begin{eqnarray*}
& \displaystyle \int_{\mathbb{R}^n} \Delta u^{\infty} \Delta \zeta -|x|^a |u^{\infty}|^{p-1}u^{\infty} \zeta = \lim\limits_{\tau \to \infty}
\displaystyle \int_{\mathbb{R}^n} \Delta u^{\tau} \Delta \zeta -|x|^a |u^{\tau}|^{p-1}u^{\tau} \zeta,  &\\[0.08cm]
& \displaystyle \int_{\mathbb{R}^n} (\Delta \zeta)^2 -p |x|^a |u^{\infty}|^{p-1} \zeta^2  = \lim\limits_{\tau \to \infty}
\displaystyle \int_{\mathbb{R}^n} (\Delta \zeta)^2 -p |x|^a |u^{\tau}|^{p-1} \zeta^2  \ge 0.&
\end{eqnarray*} Thus $u^{\infty} \in W_{loc}^{2,2}
(\mathbb{R}^n) \cap L_{loc}^{p+1}(\mathbb{R}^n)$ is a stable solution of (\ref{eq:1.1}) in $\mathbb{R}^n$. \vskip .1in

\begin{lemma}\label{eq:l4.1}
$\lim\limits_{r \to +\infty} E(r;0,u) <+\infty$.
\end{lemma}

\begin{proof}
From Theorem \ref{eq:t2.1}, we see that $E(r;0,u)$ is non-decreasing in $r$. Properties of the integral yields
\begin{equation*}
E(r; 0,u) \le \dfrac{1}{r}\int_r^{2r} E(\sigma; 0,u) d \sigma \le \dfrac{1}{r^2} \int_r^{2r} \int_t^{t+r} E(\sigma; 0,u)d \sigma dt.
\end{equation*}From (\ref{eq:2.17}). we have
\begin{align*}
& \dfrac{1}{r^2} \int_r^{2r} \int_t^{t+r} \left ( \sigma^{\frac{4(p+1)+2a}{p-1}-n} \int_{B_{\sigma}} \dfrac{1}{2}
(\Delta u)^2-\dfrac{1}{p+1}|x|^a |u|^{p+1} \right ) d \sigma dt \\
& \le \dfrac{C}{r^2} \int_r^{2r} \int_t^{t+r}  \sigma^{\frac{4(p+1)+2a}{p-1}-n} \sigma^{n-\frac{4(p+1)+2a}{p-1}}d \sigma dt \\
& \le C.
\end{align*}A simple application of H\"{o}lder's inequality and (\ref{eq:2.17}) to get
\begin{align*}
& \dfrac{1}{r^2} \int_r^{2r} \int_t^{t+r} \left (\sigma^{\frac{8+2a}{p-1}+1-n} \int_{\partial B_r} u^2 \right )d \sigma dt \\
& = \dfrac{1}{r^2} \int_r^{2r} \int_t^{t+r} \sigma^{\frac{8+2a}{p-1}+1-n}\int_{\partial B_{\sigma}}
\sigma^{-\frac{2a}{p+1}}\cdot \sigma^{\frac{2a}{p+1}}u^2 d \sigma dt \\
& \le \dfrac{1}{r^2} \int_r^{2r} \left ( \int_{B_{t+r}\backslash B_t} \left ( |x|^{\frac{8+2a}{p-1}+1-n-\frac{2a}{p+1}} \right
)^{\frac{p+1}{p-1}} \right )^{\frac{p-1}{p+1}} \left ( \int_{B_{3r}} |x|^a |u|^{p+1} \right )^{\frac{2}{p+1}} \\
& \le \dfrac{C}{r^2} \int_r^{2r} \left (|x|^{\frac{8+2a}{p-1}+1-n-\frac{2a}{p+1}+n\frac{p-1}{p+1}} \right )r^{\frac{2}{p+1}\left
[n-\frac{4(p+1)+2a}{p-1} \right ]}\\
& \le C.
\end{align*}Again applying H\"{o}lder's inequality, we find
\begin{align*}
\int_{B_r} |\nabla u|^2 \le & Cr^2 \int_{B_r} |\Delta u|^2 +Cr^{-2} \int_{B_r} u^2 \\
\le & Cr^2 \int_{B_r} |\Delta u|^2 +Cr^{-2}\left (\int_{B_r}|x|^{-\frac{2a}{p-1}} \right )^{\frac{p-1}{p+1}}
\left (\int_{B_r}|x|^a|u|^{p+1} \right )^{\frac{2}{p+1}}\\
\le & Cr^{n-\frac{8+2a}{p-1}-2}.
\end{align*}Then from the above inequality, it implies that
\begin{align*}
& \dfrac{1}{r^2} \int_r^{2r} \int_t^{t+r} \dfrac{d}{d \sigma} \left (\sigma^{\frac{8+2a}{p-1}+4-n} \int_{\partial B_{\sigma}}
|\nabla u|^2 \right ) d\sigma dt \\
& =\dfrac{1}{r^2} \int_r^{2r} \left \{(t+r)^{\frac{8+2a}{p-1}+4-n} \int_{\partial B_{t+r}} |\nabla u|^2-
t^{\frac{8+2a}{p-1}+4-n} \int_{\partial B_t} |\nabla u|^2\right \}\\
& \le \dfrac{C}{r^2} \int_{B_{3r}\backslash B_r} |x|^{\frac{8+2a}{p-1}+4-n} |\nabla u|^2 \\
& \le C
\end{align*}and
\begin{align*}
\dfrac{1}{r^2} & \int_r^{2r} \int_t^{t+r} \dfrac{\sigma^3}{2} \dfrac{d}{d \sigma} \left [
\sigma^{\frac{8+2a}{p-1}+1-n} \int_{\partial B_{\sigma}} \Big ( \dfrac{4+a}{p-1} \sigma^{-1}u+\dfrac{\partial u}{\partial r}
\Big )^2 \right ] d \sigma dt \\
= & \dfrac{1}{2r^2}\int_r^{2r} \Big \{(t+r)^{\frac{8+2a}{p-1}+4-n} \int_{\partial B_{t+r}}
\Big (\dfrac{4+a}{p-1}(t+r)^{-1}u+\dfrac{\partial u}{\partial r} \Big )^2 \\
& -
t^{\frac{8+2a}{p-1}+4-n} \int_{\partial B_r}
\Big (\dfrac{4+a}{p-1}t^{-1}u+\dfrac{\partial u}{\partial r} \Big )^2\Big \} \\
& -\dfrac{3}{2r^2} \int_r^{2r} \int_t^{t+r} \sigma^{\frac{8+2a}{p-1}+3-n} \int_{\partial B_{\sigma}}\left (\dfrac{4+a}{p-1}
\sigma^{-1} u+\dfrac{\partial u}{\partial r} \right )^2 d \sigma dt \\
\le & \dfrac{C}{r^2} \int_{B_{3r}\backslash B_r} |x|^{\frac{8+2a}{p-1}+2-n} \left (u^2 +
|x|^2 \left (\dfrac{\partial u}{\partial r} \right )^2 \right ) \\
\le & C.
\end{align*}Similarly, we can discuss the boundedness of the remaining terms in $E(r;0,u)$ and obtain the desired result.
\end{proof}

\begin{lemma}\label{eq:l4.2}
$u^{\infty}$ is homogeneous.
\end{lemma}

\begin{proof}
From the monotonicity of $E(r;0,u)$ and Lemma \ref{eq:l4.1}, it implies that for any $0<r_1<r_2<+\infty$,
\begin{equation*}
\lim\limits_{\tau \to \infty} \left [ E(\tau r_2;0,u)-E(\tau r_1; 0,u)\right ]=0.
\end{equation*}Then applying Corollary \ref{eq:c2.1} and the scaling invariance of $E$, we get
\begin{align*}
0 & =\lim\limits_{\tau \to \infty} \left [ E(r_2;0,u^{\tau})-E(r_1;0,u^{\tau})\right ] \\
& = \lim\limits_{\tau \to \infty} \int_{r_1}^{r_2} \dfrac{d }{d \sigma}E\left (\sigma; 0,u^{\tau} \right ) d \sigma \\
& \ge \lim\limits_{\tau \to \infty} c(n,p,a) \int_{B_{r_2} \backslash B_{r_1}} \frac{\left (\frac{4+a}{p-1} \sigma^{-1}
u^{\tau}+\dfrac{\partial u^{\tau}}{\partial \sigma}\right )^2}{|x|^{n-2-\frac{8+2a}{p-1}}}dx \\
& = c(n,p,a) \int_{B_{r_2} \backslash B_{r_1}} \frac{\left (\frac{4+a}{p-1} \sigma^{-1}
u^{\infty}+\dfrac{\partial u^{\infty}}{\partial \sigma}\right )^2}{|x|^{n-2-\frac{8+2a}{p-1}}}dx,
\end{align*}where $\sigma=|x|$. Therefore, we obtain
\begin{equation*}
\dfrac{4+a}{p-1}\sigma^{-1}u^{\infty}+\dfrac{\partial u^{\infty}}{\partial \sigma}=0,\quad \mbox{a.e.}
\end{equation*}A simple computation finds
\begin{equation*}
u^{\infty}(x)=|x|^{-\frac{4+a}{p-1}}u^{\infty}\left (\dfrac{x}{|x|} \right ),\quad x \in \mathbb{R}^n \backslash \{0\},
\end{equation*}i.e., $u^{\infty}$ is homogeneous.
\end{proof}

\begin{lemma}\label{eq:l4.3}
$\lim\limits_{r \to \infty} E(r;0,u)=0$.
\end{lemma}

\begin{proof}
From Lemma \ref{eq:l4.2}, it implies that $u^{\infty}$ is a homogeneous, stable solution of (\ref{eq:1.1}).
Therefore, from Theorem \ref{eq:t3.1}, we have
\begin{equation*}
u^{\infty} \equiv 0.
\end{equation*}Combining with (\ref{eq:4.1}), we find that
\begin{equation*}
\lim\limits_{\tau \to +\infty} u^{\tau} =0,\;\ \mbox{stongly}\; \mbox{in}\; \ L^2(B_5(0))
\end{equation*}
implies
\begin{equation*}
\lim\limits_{\tau \to +\infty} \int_{B_5(0)} (u^{\tau})^2 =0.
\end{equation*}Combining with the uniformly bounded of $v^{\tau}$ in $L^2(B_5(0))$, we get
\begin{equation*}
\lim\limits_{\tau \to \infty} \int_{B_5(0)} |u^{\tau}v^{\tau}| \le \lim\limits_{\tau \to \infty}
\left (\int_{B_5(0)} (u^{\tau})^2 \right )^{\frac{1}{2}} \left ( \int_{B_5(0)}(v^{\tau})^2 \right )^{\frac{1}{2}}=0.
\end{equation*}Then, it implies from (\ref{eq:2.16}) that
\begin{equation}\label{eq:4.4}
\lim\limits_{\tau \to +\infty} \int_{B_1(0)} (\Delta u^{\tau})^2 +|x|^a |u^{\tau}|^{p+1} \le
C \lim\limits_{\tau \to +\infty} \int_{B_5(0)} (u^{\tau})^2+|u^{\tau} v^{\tau}| =0.
\end{equation}

Applying the interior $L^p$-estimates yields
\begin{equation*}
\lim\limits_{\tau \to +\infty} \int_{B_2(0)} \sum\limits_{k \le 2}
|\nabla^k u^{\tau}|=0.
\end{equation*}Then, we obtain
\begin{equation*}
\int_1^2 \sum\limits_{i=1}^{\infty} \int_{\partial B_r} \sum\limits_{k\le 2}
|\nabla^k u^{\tau_i}|^2 dr \le \sum\limits_{i=1}^{\infty} \int_{B_{2r}\backslash B_r}
\sum\limits_{k \le 2} |\nabla^k u^{\tau_i}|^2 \le 1.
\end{equation*}Therefore, there exists a $\iota \in (1,2)$ such that
\begin{equation*}
\lim\limits_{\tau \to \infty} \|u^{\tau}\|_{W^{2,2}(\partial B_{\iota})}=0.
\end{equation*}Combining with (\ref{eq:4.4}) and the scaling invariance of $E(r;0,u)$, we get
\begin{equation*}
\lim\limits_{i \to \infty}E(\tau_i \iota; 0,u)=\lim\limits_{i \to \infty}E(\iota;0,u^{\tau_i})=0.
\end{equation*}Again since $\tau_i \iota \to +\infty$ and $E(r;0,u)$ is non-decreasing in $r$, we have
\begin{equation*}
\lim\limits_{r \to \infty} E(r;0,u)=0.
\end{equation*}The proof is completed.
\end{proof}

\noindent {\bf Proof of Theorem \ref{eq:t1.1}}\hspace*{5pt}
We divide the proof into three cases.\vskip .1in

\noindent {\bf Case I.} The subcritical $1 < p  <\dfrac{n+4+2a}{n-4}$.\vskip .1in

Since $p <\dfrac{n+4+2a}{n-4}$ implies $n < \dfrac{4(p+1)+2a}{p-1}$, and combining with (\ref{eq:2.17}), we find
\begin{equation*}
\int_{B_R(x)} \left ( |\Delta u|^2 +|x|^a |u|^{p+1}\right ) dx \le CR^{n-\frac{4(p+1)+2a}{p-1}} \to 0,\;\ \mbox{as}\;\ R \to +\infty.
\end{equation*}Consequently, we obtain
\begin{equation*}
u \equiv 0.
\end{equation*}\vskip .1in

\noindent {\bf Case II.} The critical $p =\dfrac{n+4+2a}{n-4}$.\vskip .1in

Utilizing the inequality (\ref{eq:2.17}) once again to find
\begin{equation*}
\int_{\mathbb{R}^n} \left ( v^2+|x|^a |u|^{p+1} \right ) dx< +\infty.
\end{equation*}Then, it implies that
\begin{equation*}
\lim\limits_{R\to +\infty} \int_{B_{2R}(x)\backslash B_R(x)} \left ( v^2+|x|^a |u|^{p+1} \right ) dx =0.
\end{equation*}From (\ref{eq:2.16}), a direct application of H\"{o}lder's inequality leads to
\begin{align*}
\int_{B_R(x)} & \left ( v^2 +|x|^a |u|^{p+1} \right ) dx \le  CR^{-4} \int_{B_{2R}(x)\backslash B_R(x)} u^2 dx
+CR^{-2} \int_{B_{2R}(x)\backslash B_R(x)} |uv| dx \\[0.06cm]
\le & CR^{-4} \left (\int_{B_{2R}(x)\backslash B_R(x)}|x|^a |u|^{p+1} dx \right )^{\frac{2}{p+1}}
\left (\int_{B_{2R}(x)\backslash B_R(x)} |x|^{-\frac{2a}{p-1}} dx \right )^{\frac{p-1}{p+1}}\\[0.06cm]
& +C \mathcal{K} R^{-2} \left ( \int_{B_{2R}(x)\backslash B_R(x)} |x|^a |u|^{p+1} dx \right )^{\frac{1}{p+1}}
\left (\int_{B_{2R}(x)\backslash B_R(x)} |x|^{-\frac{2a}{p-1}} dx \right )^{\frac{p-1}{2(p+1)}}\\[0.06cm]
\le & CR^{\left (n-\frac{4(p+1)+2a}{p-1} dx \right )\frac{p-1}{p+1}}\left (
\int_{B_{2R}(x)\backslash B_R(x)}|x|^a |u|^{p+1} dx \right )^{\frac{2}{p+1}} \\[0.06cm]
& +C \mathcal{K} R^{\left (n-\frac{4(p+1)+2a}{p-1} dx \right )\frac{p-1}{2(p+1)} }\left (
\int_{B_{2R}(x)\backslash B_R(x)} |x|^a |u|^{p+1} dx \right )^{\frac{1}{p+1}},
\end{align*}where $\mathcal{K}=\left (\int_{B_{2R}(x)\backslash B_R(x)} v^2 dx \right )^{\frac{1}{2}}$.
Since $p =\dfrac{n+4+2a}{n-4}$, the right side of the above inequality tends to $0$ as $R \to +\infty$.
So we get
\begin{equation*}
u \equiv 0.
\end{equation*}\vskip .05in

\noindent {\bf Case III.} The supercritical
$\dfrac{n+4+2a}{n-4} < p < p_a(n) $.\vskip .08in

The smoothness of $u$ implies that
\begin{equation*}
\lim\limits_{r \to 0} E(r;0,u)=0.
\end{equation*}From the monotonicity
of $E(r;0,u)$ and Lemma \ref{eq:l4.3}, it implies that
\begin{equation*}
E(r;0,u)=0,\;\; \mbox{for}\; \mbox{all}\; r>0.
\end{equation*}Then, from Corollary \ref{eq:c2.1}, $u$ is homogeneous. Therefore from Theorem \ref{eq:t3.1},
we obtain
\begin{equation*}
u\equiv 0. \eqno \square
\end{equation*}

\section{Proof of Theorem \ref{eq:t1.2}.}

In this section, we study the finite Morse index solutions of (\ref{eq:1.1})
by the use of the {\it Pohozaev-type identity}, {\it monotonicity
formula} and {\it blowing down} sequence.\vskip .1in

A basic ingredient of the proof of the subcritical case in Theorem \ref{eq:t1.2} is the following {\it
Pohozaev-type identity}.

\begin{lemma}\label{eq:l5.1}
we have the equality
\begin{align}\label{eq:5.1}
& \int_{B_R} \left ( \dfrac{n-4}{2} |\Delta u|^2-\dfrac{n+a}{p+1} |x|^a |u|^{p+1} \right ) dx \nonumber \\[0.06cm]
&= \int_{\partial B_R} \left ( \dfrac{R}{2} (\Delta u)^2 -\dfrac{1}{p+1} R^{1+a}|u|^{p+1}+R \dfrac{\partial u}{\partial r}
\dfrac{\partial \Delta u}{\partial r}-\Delta u \dfrac{\partial (x\cdot \nabla u)}{\partial r} \right ) dS.
\end{align}
\end{lemma}\vskip .1in

Applying the doubling lemma in \cite[Lemma 5.1]{Polacik},
we get the following estimates.

\begin{lemma}\label{eq:l5.2}
Let $u$ be a finite Morse index solution of (\ref{eq:1.1}). Then
there exist constants $C$ and $R^*$ such that
\begin{equation}\label{eq:5.2}
|u(x)|\le C|x|^{-\frac{4+a}{p-1}}, \;\ \mbox{for}\;\ \mbox{all}\; x \in B_{R^*}^c,
\end{equation}and
\begin{equation}\label{eq:5.3}
\sum\limits_{k\le 3} |x|^{\frac{4+a}{p-1}+k} |\nabla^ku(x)|\le C, \;\ \mbox{for}\; \mbox{all}\;\ x \in B_{3R^*}^c
\end{equation}
\end{lemma}

\begin{proof}The inequality (\ref{eq:5.2}) can be deduced as in \cite[Lemma 5.1]{Davila}.

Next, we only prove the inequality (\ref{eq:5.3}).
Take arbitrarily $\widetilde{x}$ with $|\widetilde{x}| >3R^*$ and $\tau =\dfrac{|\widetilde{x}|}{2}$, and denote
\begin{equation*}
\omega(x):=\tau^{\frac{4+a}{p-1}} u(\widetilde{x}+\tau x).
\end{equation*}From (\ref{eq:5.2}), it implies that for any $x \in B_1(0)$
\begin{equation*}
|\omega(x)|\le C \tau^{\frac{4+a}{p-1}} (|\widetilde{x}+\tau x|)^{-\frac{4+a}{p-1}} \le C_1.
\end{equation*}Then we get from the standard elliptic estimates that
\begin{equation*}
\sum\limits_{k \le 3} |\nabla^k \omega(0)| \le C_2.
\end{equation*}Noting that $\nabla^k \omega(x) =\tau^{\frac{4+a}{p-1}+k} \nabla^k u(\widetilde{x}+\tau x)$.
Therefore we conclude that
\begin{equation*}
\sum\limits_{k \le 3} |x|^{\frac{4+a}{p-1}+k}|\nabla^k u(x)| \le C_2,
\end{equation*}for all $x \in B_{3R^*}(0)^c$.
\end{proof}

\noindent {\bf Proof of Theorem \ref{eq:t1.2}.}\hspace*{5pt}
The proof consists of three cases.\vskip .1in

\noindent {\bf Case I.} The subcritical $1< p< \dfrac{n+4+2a}{n-4}$.\vskip .1in

From (\ref{eq:5.2}) and (\ref{eq:5.3}), we get the estimate of the right side in (\ref{eq:5.1}),
\begin{equation*}
\int_{\partial B_R} \dfrac{R}{2} (\Delta u)^2+\dfrac{R^{1+a}}{p+1} |u|^{p+1}+R \dfrac{\partial u}{\partial r}
\dfrac{\partial \Delta u}{\partial r}+\left |\Delta u \dfrac{\partial (x \cdot \nabla u) }{\partial r} \right |\to 0,\;\ \mbox{as}\;
R \to +\infty.
\end{equation*}

On the other hand, since $u$ is stable outside a compact set $\Omega \subset \mathbb{R}^n$, we can take
a test function $\zeta_R \in
C^4_0 (\mathbb{R}^n \backslash \Omega)$ for $R>R^*+4$ and $\Omega \subset B_{R^*}$,
\begin{equation*}
\zeta_R(x)=
\begin{cases}
0, & \mbox{if}\; |x|<R^*+1\; \mbox{or}\; |x|>2R, \\
1, & \mbox{if}\; R^*+2< |x|<R.
\end{cases}
\end{equation*}which satisfies $0 \le \zeta_R \le 1$, $\|\nabla^i \zeta_R\|_{L^{\infty}(B_{2R}\backslash B_R)} \le \dfrac{C}{R^i}$
and $\|\nabla^i \zeta_R\|_{L^{\infty}(B_{R^*+2} \backslash B_{R^*+1}) } \le C_{R^*}$, for $i=1,2,3,4$.
Then from Lemma \ref{eq:l2.2}, we have
\begin{equation*}
\int_{R^*+2<|x|<R} \left (|\Delta u|^2+|x|^a |u|^{p+1} \right ) dx \le C_{R^*}+C R^{n-\frac{4(p+1)+2a}{p-1}}.
\end{equation*}Again since $n<\dfrac{4(p+1)+2a}{p-1}$, we obtain
\begin{equation*}
\int_{\mathbb{R}^n} \left [ (\Delta u)^2 +|x|^a |u|^{p+1} \right ] dx <+\infty.
\end{equation*}Taking limit in (\ref{eq:5.1}), we obtain
\begin{equation}\label{eq:5.4}
\int_{\mathbb{R}^n} \left [ \dfrac{n-4}{2} |\Delta u|^2-\dfrac{n+a}{p+1} |x|^a |u|^{p+1}\right ]dx=0.
\end{equation}\vskip .1in

Now, we claim that
\begin{equation}\label{eq:5.5}
\int_{\mathbb{R}^n} |\Delta u|^2 dx= \int_{\mathbb{R}^n}  |x|^a |u|^{p+1} dx.
\end{equation}

Indeed, multiply the equation (\ref{eq:1.1}) with $u\zeta_R$ for
$\zeta_R \in C^4_0(B_{2R})$ which satisfies $0 \le \zeta_R \le 1$,
$\|\nabla^i \zeta_R\|_{L^{\infty}} \le \dfrac{C}{R^i}$, for $i=1,2,3,4$, and
\begin{equation*}
\zeta_R(x)=
\begin{cases}
1, & \mbox{if}\; |x|<R,\\
0, & \mbox{if}\; |x|>2R.
\end{cases}
\end{equation*}A simple computation implies
\begin{equation*}
\int_{B_R} \left ( |x|^a |u|^{p+1}-(\Delta u)^2\right ) \zeta_R dx=\int_{B_R} \left ( u \Delta u \Delta \zeta_R+
2\Delta u\nabla u \cdot \nabla \zeta_R \right ) dx:=S_1(R)+S_2(R).
\end{equation*}We may use H\"{o}lder's inequality in $S_1(R)$ and $S_2(R)$ to obtain
\begin{align*}
|S_1(R)| \le & R^{-2} \int_{B_R} |\Delta u| \left ( |x|^{\frac{a}{p+1}}|u| \right )|x|^{-\frac{a}{p+1}} \\
\le & R^{-2}\left ( \int_{B_{2R}} (\Delta u)^2 \right )^{\frac{1}{2}}\left ( \int_{B_{2R}}
|x|^a |u|^{p+1} \right )^{\frac{1}{p+1}} \left ( \int_{B_{2R}} |x|^{-\frac{2a}{p-1}}
\right )^{\frac{p-1}{2(p+1)}} \\
\le & C R^{\frac{n(p-1)}{2(p+1)} -\frac{a}{p+1}-2} \left ( \int_{B_{2R}} (\Delta u)^2 \right )^{\frac{1}{2}}
\left ( \int_{B_{2R}}
|x|^a |u|^{p+1} \right )^{\frac{1}{p+1}} \\
\le & CR^{\frac{p-1}{2(p+1)}\left (n-\frac{4(p+1)+2a}{p-1}\right )}
\end{align*}and
\begin{align*}
|S_2(R)| = & \int_{B_{2R}}|\Delta u|\cdot |\nabla u| \cdot |\nabla \zeta_R| \le \left ( \int_{B_{2R}}(\Delta u)^2 \right )^{\frac{1}{2}}
\left ( \int_{B_{2R}}|\nabla u|^2 |\nabla \zeta_R|^2 \right )^{\frac{1}{2}} \\
= & \left (\int_{B_{2R}} (\Delta u)^2 dx \right )^{\frac{1}{2}} \left (\int_{B_{2R}} u(- \Delta u) |\nabla \zeta_R|^2+
\dfrac{1}{2} \int_{B_{2R}} u^2 \Delta ( |\nabla \zeta_R|^2) \right )^{\frac{1}{2}}\\
\le & C \left (\int_{B_{2R}} (\Delta u)^2 dx \right )^{\frac{1}{2}} \left (\int_{B_{2R}} |u||\Delta u| |\nabla \zeta_R|^2
\right )^{\frac{1}{2}} \\
& +C \left (\int_{B_{2R}} (\Delta u)^2 dx \right )^{\frac{1}{2}}\left (
\int_{B_{2R}} |x|^a |u|^{p+1} \right )^{\frac{1}{p+1}} \left (
\int_{B_{2R}} |x|^{-\frac{2a}{p-1}} \right )^{\frac{p-1}{2(p+1)}} \\
\le & CR^{\left [n-\frac{4(p+1)+2a}{p-1} \right ]\frac{p-1}{2(p+1)}}.
\end{align*}In the above, we use the results in (\ref{eq:2.15}), (\ref{eq:5.2}) and (\ref{eq:5.3}).
Since $n < \dfrac{4(p+1)+2a}{p-1}$, we get
\begin{equation*}
\lim\limits_{R \to +\infty}S_1 (R)=0\;\ \mbox{and}\;\
\lim\limits_{R \to \infty} S_2(R)=0.
\end{equation*}Thus, the claim (\ref{eq:5.5}) holds.

Combining (\ref{eq:5.4}) with (\ref{eq:5.5}), this leads to
\begin{equation*}
\left (\dfrac{n-4}{2}-\dfrac{n+a}{p+1} \right )\int_{\mathbb{R}^n} |u|^{p+1} dx =0.
\end{equation*}Thus we get that
\begin{equation*}
u\equiv 0.
\end{equation*}\vskip .1in

\noindent {\bf Case II.} The critical $n=\dfrac{4(p+1)+2a}{p-1}$.\vskip .1in

Since $u$ is stable outside $B_{R^*}$, we adopt the similar argument as in the subcritical case and find
\begin{equation*}
\int_{B_R \backslash B_{3R^*}}\left [ (\Delta u)^2 +|x|^{\alpha}|u|^{p+1} \right ] dx\le C, \;\ \mbox{for}\; R>3R^*
\end{equation*}and
\begin{equation*}
\int_{\mathbb{R}^n} \left [ (\Delta u)^2 +|x|^{\alpha} |u|^{p+1} \right ] dx < +\infty.
\end{equation*}The elliptic regularity theory implies
\begin{equation*}
\lim\limits_{R \to \infty} \int_{B_{2R}\backslash B_R} R^{-1} |\nabla u|
+R^{-2} |u|=0.
\end{equation*}Therefore, it is easy to verify that
\begin{equation*}
\int_{R^n} (\Delta u)^2 -|x|^{\alpha} |u|^{p+1}=0.
\end{equation*}\vskip .1in

\noindent {\bf Case III.} The supercritical $\dfrac{n+4+2a}{n-4}<p<p_a(n)$. \vskip .1in

\noindent {\it Claim I.}
There exists a constant $C$ such that for all $r>3R^*$, $E(r;0,u)\le C$.\vskip .1in

Indeed, applying the inequality (\ref{eq:5.2}) and (\ref{eq:5.3}), we obtain
\begin{align*}
E(r;0,u) \le & Cr^{\frac{4(p+1)+2a}{p-1}-n} \int_{B_r} (\Delta u)^2 +|x|^a |u|^{p+1} \\
& +C r^{\frac{8+2a}{p-1}+1-n} \int_{\partial B_r} u^2 +Cr^{\frac{8+2a}{p-1}+3-n} \int_{\partial B_r} |\nabla u|^2 \\
& +Cr^{\frac{8+2a}{p-1}+2-n} \int_{\partial B_r} |u||\nabla u|+Cr^{\frac{8+2a}{p-1}+4-n} \int_{\partial B_r} |\nabla u|
|\nabla^2 u| \\
\le & C,
\end{align*}for all $r >3R^*$, where $C$ does not depend on $r$.

Now, we can apply Theorem \ref{eq:t2.1} to get
\begin{align*}
\dfrac{d}{dr} E(r;0,u) \ge & c(n,p,a) r^{-n+2+\frac{8+2a}{p-1}} \int_{\partial B_r} \left (
\dfrac{4+a}{p-1} r^{-1} u+\dfrac{\partial u}{\partial r} \right )^2 \\
= & c(n,p,a) \int_{\partial B_r} \dfrac{ \left ( \frac{4+a}{p-1} r^{-1}u+\frac{\partial u}{\partial r}
\right )^2}{r^{n-2-\frac{8+2a}{p-1}}}.
\end{align*}Integrating the above inequality from $3R^*$ to $+\infty$ in both sides and combining with {\it Claim I}, we find
\begin{equation}\label{eq:5.7}
\int_{B_{3R^*}^c} \dfrac{ \left ( \frac{4+a}{p-1} \sigma^{-1}u+\frac{\partial u}{\partial \sigma}
\right )^2}{|x|^{n-2-\frac{8+2a}{p-1}}}< +\infty.
\end{equation}\vskip .1in

\noindent {\it Claim II.}
$\lim\limits_{r \to +\infty} E(r;0,u)=0$.\vskip .1in

Indeed, for $\tau>0$, we define a {\it blowing down} sequence
\begin{equation*}
u^{\tau}(x):=\tau^{\frac{4+a}{p-1}} u(\tau x).
\end{equation*}It implies from Lemma \ref{eq:l5.2} that $u^{\tau}$ is uniformly
bounded in $C^5 \left (B_r(0)\backslash B_{1/r}(0) \right )$ for any
fixed $r>1$, and $u^{\tau}$ is stable outside $B_{r/\tau} (0)$. Then there exists a function $u^{\infty}$ in $C^4 \left
(\mathbb{R}^n \backslash \{0\} \right )$ such that $u^{\infty}$ is a stable solution of (\ref{eq:1.1}) in $\mathbb{R}^n \backslash \{0\}$.
For any $r>1$, we get from (\ref{eq:5.7}) that
\begin{align*}
& \int_{B_{r} \backslash B_{1/r}} \dfrac{ \left ( \frac{4+a}{p-1} \sigma^{-1}u^{\infty}+\frac{\partial u^{\infty}}{\partial \sigma}
\right )^2}{|x|^{n-2-\frac{8+2a}{p-1}}} \\
& =\lim\limits_{\tau \to \infty} \int_{B_{r} \backslash B_{1/r}} \dfrac{ \left ( \frac{4+a}{p-1} \sigma^{-1}u^{\tau}+\frac{\partial u^{\tau}}{\partial \sigma}
\right )^2}{|x|^{n-2-\frac{8+2a}{p-1}}}\\
& =\lim\limits_{\tau \to \infty} \int_{B_{\tau r} \backslash B_{\tau /r}} \dfrac{ \left ( \frac{4+a}{p-1} \sigma^{-1}u+\frac{\partial u}{\partial \sigma}
\right )^2}{|x|^{n-2-\frac{8+2a}{p-1}}} \\
& =0.
\end{align*}From Corollary \ref{eq:c2.1}, we conclude that
$u^{\infty}$ is a homogeneous, stable solution of (\ref{eq:1.1}).
Then we get from Theorem \ref{eq:t3.1} that
\begin{equation*}
u^{\infty} \equiv 0.
\end{equation*}
Consequently, from the definition of the blowing down sequence $u^{\tau}(x)$ and the argument as in
(\ref{eq:5.3}), we get
\begin{equation*}
\lim\limits_{|x| \to \infty} |x|^{\frac{4+a}{p-1}} |u(x)|=0,
\end{equation*}and
\begin{equation*}
\lim\limits_{|x|\to \infty} \sum\limits_{k \le 4} |x|^{\frac{4+a}{p-1}+k} |\nabla^k u(x)|=0.
\end{equation*}
For any $\epsilon >0$ and $R_0>0$, we find
\begin{equation*}
\sum\limits_{k \le 4} |x|^{\frac{4+a}{p-1}+k} |\nabla^k u(x)| \le \epsilon
\end{equation*}for all $|x|>R_0$. Then for $r \gg R_0$, we have
\begin{align*}
E(r;0,u) \le & Cr^{\frac{4(p+1)+2a}{p-1}-n} \left \{ \int_{B_{R_0}(0) \cup [B_r(0) \backslash B_{R_0}(0)]}
\left [ (\Delta u)^2 +|x|^a |u|^{p+1} \right ]
 \right \} \\
& +C\epsilon r^{\frac{8+2a}{p-1}+1-n} \int_{\partial B_r(0)} |x|^{-\frac{8+2a}{p-1}} \\
\le & C(R_0) \left ( r^{\frac{4(p+1)+2a}{p-1}-n} +\epsilon \right ).
\end{align*}From the inequality $n > \dfrac{4(p+1)+2a}{p-1}$ and the arbitrary of $\epsilon$, we conclude that {\it Claim II}
holds.\vskip .1in

From the smoothness of $u$, it is easy to see that
\begin{equation*}
\lim\limits_{r \to 0} E(r;0,u)=0.
\end{equation*}Hence, from {\it Claim II} and the monotonicity of $E$, we get
\begin{equation*}
u \equiv 0.\eqno \square
\end{equation*}

\vskip .1in

\noindent {\bf Acknowledge:} We thank Department of mathematics, The Chinese University of Hong Kong for its kind hospitality.

\end{document}